\documentclass[12pt]{extarticle}
\usepackage{amsmath,amsfonts,amsthm,latexsym}
\usepackage{enumerate,etoolbox}
\usepackage{multirow}
\usepackage[colorlinks=true, urlcolor=blue, linkcolor=blue, citecolor=blue]{hyperref}

\numberwithin{equation}{section}
\newtheorem{thm}{Theorem}
\numberwithin{thm}{section}

\oddsidemargin \evensidemargin
\newtheorem{prop}[thm]{Proposition}
\newtheorem{lemma}[thm]{Lemma}
\newtheorem{cor}[thm]{Corollary}
\newtheorem{rmk}[thm]{Remark}

\date{}
\author{Hwangrae Lee}
\title{Power Sum Decompositions of \hbox{Elementary Symmetric Polynomials}}

\begin{document}

\maketitle

\begin{abstract}
We bound the tensor ranks of elementary symmetric polynomials,
and we give explicit decompositions into powers of linear forms.
The bound is attained when the degree is odd.

\end{abstract}

\section{Introduction} 
Given a form $F \in \mathbb{C}[x_1,\dots,x_n]$ of degree $d$, the {\em symmetric tensor rank} (rank in short) of $F$ is the least integer $s$ such that $f = \sum_{i=1}^s L_i^d$, where the $L_i$'s are linear forms.
For a generic form, the rank is known for any $n$ and $d$ by a work of Alexander and Hirschowitz \cite{AH}, and a simple proof is proposed by Chandler \cite{Chandler}.

On the contrary, only a few cases is known for that of specific forms \cite{CCG, CCCGW, LT, TW}.
To describe one of them, we define an index-membership function $\delta$; 
for an integer set $I$ and an integer $i$, define $\delta(I,i) = -1 $ if $i \in I$, or $1$ otherwise.
A decomposition of the monic square-free monomial $\sigma_{n,n}$ in $n$ variables was given by Fischer \cite{Fischer}. He showed that 
\begin{equation}\label{monomial}
2^{n-1}n!\cdot\sigma_{n,n} = \sum_{I \subset [n]\setminus \{1\}} (-1)^{|I|}
(x_1 + \delta(I,2) x_2 + \dots +\delta(I,n) x_n)^n,
\end{equation}
where $[n]$ denotes the set $\{1,2,\dots,n\}$, which will be used though this paper.
For a general monomial, such a decomposition is given in \cite[Corollary 3.8]{CCG} so that \eqref{monomial} is a special case.

We extend the decomposition \eqref{monomial} for general elementary symmetric polynomials $\sigma_{d,n}$. As an example, we have
$$
24\sigma_{3,5}(a,b,c,d,e) = abc+abd+abe+acd+ace+ade+bcd+bce+bde+cde
$$
$$
=  3(a+b+c+d+e)^3-(-a+b+c+d+e)^3-(a-b+c+d+e)^3
$$
$$
-(a+b-c+d+e)^3-(a+b+c-d+e)^3-(a+b+c+d-e)^3,
$$
and it gives an upper bound $\text{rank}(\sigma_{3,5}) \leq 6$.
\newpage
In general, such a decomposition provides an upper bound
$$\text{rank}(\sigma_{d,n}) \leq \sum_{i=0}^{\lfloor d/2\rfloor} \binom{n}{i}$$
for the rank of $\sigma_{d,n}$ (Corollary \ref{bound1}, Corollary \ref{bound2}).

This paper is organized as follows. In section 2 we present a power sum decomposition of $\sigma_{d,n}$ for odd degree case. Theorem \ref{decomp} gives an analogue of \eqref{monomial}.

In Section 3 we observe a structure of the catalecticants of $\sigma_{d,n}$. Lemma \ref{fullrank} says that each catalecticant matrix is essentially full rank, in the sence that it can be refined to a full rank matrix after removing all zero rows and zero columns. Theorem \ref{main} tell us that the lower bound derived by Lemma \ref{fullrank} matches the number of components in the decomposition given in Section 2, hence we get the rank of $\sigma_{d,n}$ for odd $d$.

Section 4 discusses the even degree case. A power sum decomposition of $\sigma_{d,n}$ can be obtained from that of $\sigma_{d+1,n}$ (Theorem \ref{decomp:odd}), or it can be derived directly (Remark \ref{direct}). By Corollary \ref{bound2} we see that upper and lower bounds are not the same, and we shortly explain why it seems to be hard to improve these bounds.

\section{An upper bound}
Through this paper, our main object is the elementary symmetric polynomial $\sigma_{d,n}$.
It is defined by the sum of all square-free monomials. In symbols,
$$\sigma_{d,n} = \sum_{I \subset [n], |I| = d} \prod_{i \in I} x_i.$$

In this section we only focus on the odd degree case. Let $d=2k+1$ be odd, and consider the monomial case $n=d$. The first step is to write \eqref{monomial} into a symmetric way. Let $L_I = \delta(I,1)x_1 + \delta(I,2)x_2+\cdots+\delta(I,n)x_n$ for some $I \subset [n]$, recalling that $\delta(I,i) = \pm 1$ and takes $-1$ if and only if $i \in I$. When $|I| \geq k+1$, we use $-(-L_I)^n = -(L_{I^c})^n$ instead of $(L_I)^n$ in the expression \eqref{monomial}. Now all the linear terms appearing in \eqref{monomial} have less than half as many possible minus signs, and the coefficient of $x_1$ can be $-1$. After consideration of the power of $-1$, we get
\begin{equation}\label{sym}
2^{n-1}n! \cdot \sigma_{n,n}  
= \sum_{I \subset [n], |I| \leq k} (-1)^{|I|}
(\delta(I,1) x_1 + \delta(I,2) x_2 + \dots +\delta(I,n) x_n)^n.
\end{equation}

Using this symmetric representation as the initial step, we can prove following theorem using induction.

\newpage

\begin{thm}\label{decomp}
Let $d=2k+1$ be odd, and $n \geq d$. Then the elementary symmetric polynomial $\sigma_{d,n}$ admits the power sum decomposition
$$
2^{d-1}d! \cdot \sigma_{d,n} = 
$$
\begin{equation}\label{psdec}
\sum_{I \subset [n], |I| \leq k} (-1)^{|I|} \binom{n-k-|I|-1}{k-|I|} (\delta(I,1) x_1 + \delta(I,2) x_2 + \dots +\delta(I,n) x_n)^d
\end{equation}
where $\delta(I,i) = -1 $ if $i \in I$, or $1$ otherwise.
\end{thm}

\begin{proof}
Write $F_{d,n}$ for the expression \eqref{psdec}. Putting $x_n=0$, we get 
\begin{small}
\begin{align*}
&F_{d,n}(x_1,\dots,x_{n-1},0)\\
=&\sum_{I \subset [n-1], |I| \leq k} (-1)^{|I|} \left(\binom{n-k-|I|-1}{k-|I|} - \binom{n-k-|I|-2}{k-|I|-1} \right) (\delta(I,1) x_1 + \dots +\delta(I,n-1) x_{n-1})^d\\
=&\sum_{I \subset [n-1], |I| \leq k} (-1)^{|I|} \binom{(n-1)-k-|I|-1}{k-|I|} (\delta(I,1) x_1 + \dots +\delta(I,n-1) x_{n-1})^d
\,=\, F_{d,n-1} 
\end{align*}
\end{small}
Recursively we have $F_{d,n}(x_1,\dots,x_d,0,\dots,0) = F_{d,d}$. By \eqref{sym} we obtain
$$
F_{d,d}=\sum_{I \subset [d], |I| \leq k} (-1)^{|I|} \binom{k-|I|}{k-|I|} (\delta(I,1) x_1 + \delta(I,2) x_2 + \dots +\delta(I,d) x_d)^d = 2^{d-1}d!\cdot \sigma_{d,d}.
$$
It shows that $F_{d,n}$ consists of square-free monomials only. Using the symmetry of $F_{d,n}$ we conclude $$F_{d,n} = 2^{d-1}d!\cdot\sigma_{d,n}.$$
\end{proof}

Counting the number of summands, we get an upper bound for the rank of $\sigma_{d,n}$.

\begin{cor}\label{bound1}
For $d$ odd, the rank of $\sigma_{d,n}$ is bounded by
$$\text{\em rank}(\sigma_{d,n}) \leq \sum_{i=0}^{(d-1)/2} \binom{n}{i}.$$
\end{cor}

Another consequence is a summation identity, as done in \cite{Fischer}. 
For $n \geq 2k+1$ the identity
\begin{equation}\label{comb}
\sum_{i=0}^k (-1)^i \binom{n-k-1-i}{k-i}\binom{n}{i}(n-2i)^{2k+1} = \frac{2^{2k}n!}{(n-2k-1)!}
\end{equation}
can be obtained by choosing $d=2k+1$ and all $x_i=1$ in the equation in Theorem \ref{psdec}.

\section{A lower bound}
Let $S = \mathbb{C}[\frac{\partial}{\partial x_1},\dots,\frac{\partial}{\partial x_n}]$ be the ring of differential operations with constant coefficients.
It naturally acts on \hbox{$R = \mathbb{C}[x_1,\dots,x_n]$} by differentiation. For a form $F$ in $R$ its apolar ideal $F^{\bot}$ is defined by the annihilator of $F$ in $S$. 

\begin{thm}[Apolarity Lemma]
For a degree $d$ form $F \in R_d$ there is a power sum decomposition
$$
F = \sum_{i=1}^s L_i^d, \qquad L_i \text{ linear}
$$
if and only if there exists a set of $s$ distinct points in $\mathbb{P}(S_1)$ whose defining ideal is contained in $F^{\bot}$.
\end{thm}

The apolarity lemma plays a key role in computations of symmetric tensor rank, especially for lower bounds. Many (possibly all) known lower bounds, given in \cite{CCCGW, LT, RS} for instance, are related to the apolar ideal or at least catalecticants.

The $r$-th catalecticant of a given form $F \in R_d$ is a linear map $\phi_r : S_r \rightarrow R_{d-r}$ given by $\phi(g) = gF$. Using monomial basis, $\phi_r$ can be written as an $\binom{n+d-r-1}{d-r} \times \binom{n+d-1}{d}$ matrix $M_r$. Studying catalecticants and apolar ideals are essentially same, by the relation $F^{\bot} = \bigcup \ker \phi_r$ or equivalently $\ker \phi_r = (F^{\bot})_r$. Note that it implies $\text{Hilb}(S/F^{\bot}, r)=\text{rank}(M_r)$

For our case $F=\sigma_{d,n}$, the matrix $M_r$ has many zero columns and zero rows. A column or a row of $M_r$ is nonzero if and only if its index is square-free. Removing all zero columns and zero rows, we get a $\binom{n}{d-r} \times \binom{n}{r}$ submatrix $\widetilde{M_r}$ of $M_r$, whose indices are square-free monomials of corresponding degrees. Collecting subindices of monomials, we regards the indices of rows and columns as a subset of $[n]$. The matrix $\widetilde{M_r}$ is binary with $(\widetilde{M_r})_{I,J} = 1$ if and only if $I$ and $J$ are disjoint.

For instance, the second catalecticant of $\sigma_{4,5}$ is given by
\begin{small}
$$M_2=
\bordermatrix{
  &\bf11 & 12 & 13 & 14 & 15 & \bf22 & 23 & 24 & 25 & \bf33 & 34 & 35 & \bf44 & 45 & \bf55\cr
\bf11 & \bf0&\bf0&\bf0&\bf0&\bf0&\bf0&\bf0&\bf0&\bf0&\bf0&\bf0&\bf0&\bf0&\bf0&\bf0\cr
12 & \bf0&0&0&0&0&\bf0&0&0&0&\bf0&1&1&\bf0&1&\bf0\cr
13 & \bf0&0&0&0&0&\bf0&0&1&1&\bf0&0&0&\bf0&1&\bf0\cr
14 & \bf0&0&0&0&0&\bf0&1&0&1&\bf0&0&1&\bf0&0&\bf0\cr
15 & \bf0&0&0&0&0&\bf0&1&1&0&\bf0&1&0&\bf0&0&\bf0\cr
\bf22 & \bf0&\bf0&\bf0&\bf0&\bf0&\bf0&\bf0&\bf0&\bf0&\bf0&\bf0&\bf0&\bf0&\bf0&\bf0\cr
23 & \bf0&0&0&1&1&\bf0&0&0&0&\bf0&0&0&\bf0&1&\bf0\cr
24 & \bf0&0&1&0&1&\bf0&0&0&0&\bf0&0&1&\bf0&0&\bf0\cr
25 & \bf0&0&1&1&0&\bf0&0&0&0&\bf0&1&0&\bf0&0&\bf0\cr
\bf33 & \bf0&\bf0&\bf0&\bf0&\bf0&\bf0&\bf0&\bf0&\bf0&\bf0&\bf0&\bf0&\bf0&\bf0&\bf0\cr
34 & \bf0&1&0&0&1&\bf0&0&0&1&\bf0&0&0&\bf0&0&\bf0\cr
35 & \bf0&1&0&1&0&\bf0&0&1&0&\bf0&0&0&\bf0&0&\bf0\cr
\bf44 & \bf0&\bf0&\bf0&\bf0&\bf0&\bf0&\bf0&\bf0&\bf0&\bf0&\bf0&\bf0&\bf0&\bf0&\bf0\cr
45 & \bf0&1&1&0&0&\bf0&1&0&0&\bf0&0&0&\bf0&0&\bf0\cr
\bf55 & \bf0&\bf0&\bf0&\bf0&\bf0&\bf0&\bf0&\bf0&\bf0&\bf0&\bf0&\bf0&\bf0&\bf0&\bf0\cr}.
$$
\end{small}
\newpage
\noindent Its rows and columns are indexed by $2$-subsets of $\{1,\dots,5\}$, and a row (resp. a column) indexed by $ij$ corresponds to the monomial $x_i x_j$ (resp. $\frac{\partial^2}{\partial x_i x_j}$). Removing zero rows and zero columns indexed by $\bf11, 22, \dots, 55$, we obtain
$$\widetilde{M_2}=
\bordermatrix{
  & 12 & 13 & 14 & 15 & 23 & 24 & 25 & 34 & 35 & 45\cr
12 & 0&0&0&0&0&0&0&1&1&1\cr
13 & 0&0&0&0&0&1&1&0&0&1\cr
14 & 0&0&0&0&1&0&1&0&1&0\cr
15 & 0&0&0&0&1&1&0&1&0&0\cr
23 & 0&0&1&1&0&0&0&0&0&1\cr
24 & 0&1&0&1&0&0&0&0&1&0\cr
25 & 0&1&1&0&0&0&0&1&0&0\cr
34 & 1&0&0&1&0&0&1&0&0&0\cr
35 & 1&0&1&0&0&1&0&0&0&0\cr
45 & 1&1&0&0&1&0&0&0&0&0\cr}
$$
and one may check that is has full rank. Next lemma shows that it is true in general, including the case that the resulting $\widetilde{M_r}$ is not a square matrix.

\begin{lemma}\label{fullrank}
For all $r, d,$ and $n$ with $r \leq d \leq n$, the matrix $\widetilde{M_r}$ is of full rank.
\end{lemma}
\begin{proof}
Since $\widetilde{M_r} = (\widetilde{M}_{d-r})^{\rm T}$, we may assume $r \leq d-r$. Note that it forces $\binom{n}{d-r} \geq \binom{n}{r}$ and $2r \leq n$. Our goal is to show ${\rm rank}(\widetilde{M_r}) = \binom{n}{r}$.

Define an $\binom{n}{r} \times \binom{n}{r}$ binary matrix $D_r^n$ whose columns and rows are indexed by $r$-subsets of $[n]$, and $(D_r^n)_{I,J} = 1$ if and only if $I$ and $J$ are disjoint. In \cite[Example 2.12]{KN}, the matrix $D_r^n$ is shown to be invertible when $2r \leq n$. 

We claim that the row space of $D_r^n$ is a subspace of the row space of $\widetilde{M_r}$. Pick a row vector $v_I$ of $D_r^n$ indexed by an $r$-subset $I$. Consider a row vector
$$
w_I = \sum_{J \supset I} (\widetilde{M_r})_{J,*}
$$
where each summand is the row of $\widetilde{M_r}$ indexed by $J$. By the symmetry of indices, $w_I = c\cdot v_I$ for some integer $c$. Since $\widetilde{M_r}$ has no zero row, $c \neq 0$ and we deduce the claim.
\end{proof}

Since zero rows and zero columns do not contribute to the matrix rank, the following corollary is an immediate consequence.

\begin{cor}\label{hilb}
The Hilbert function of $S/(\sigma_{d,n})^{\bot}$ is given by
$$
\text{\em Hilb}(S/(\sigma_{d,n})^{\bot},r)=
\left\{
\begin{array}{cl}
\binom{n}{r} & \text{ if } r \leq \lfloor d/2 \rfloor \\
\binom{n}{d-r} & \text{ if } r > \lfloor d/2 \rfloor 
\end{array}
\right.
.$$
\end{cor}

It gives a lower bound $\text{rank}(\sigma_{d,n}) \geq \max \{\text{Hilb}(S/(\sigma_{d,n})^{\bot},r)\} = \binom{n}{\lfloor d/2 \rfloor}$ by apolarity lemma, but we can go further.
For any nonzero linear form $L \in S$, the number $\dim_{\mathbb{C}} S/((F^{\bot} : L)+(L))$ gives a lower bound for the rank of $F$ \cite[Theorem 3.3]{CCCGW}.

Take $L = \frac{\partial}{\partial x_n}$. Then it is easy to see that
$$
((\sigma_{d,n})^{\bot}:L)+L = S\cdot (\sigma_{d-1,n-1})^{\bot}+L.
$$
Here, $(\sigma_{d-1,n-1})^{\bot}$ can be computed in either $S$ or $S'=\mathbb{C}[\frac{\partial}{\partial x_1},\dots, \frac{\partial}{\partial x_{n-1}}] \subset S$. Since $S/L = S'$, we conclude 
$$\text{rank}(\sigma_{d,n}) \geq \dim_{\mathbb{C}}S'/(\sigma_{d-1,n-1})^{\bot}.$$

\begin{thm}\label{main}
For $d$ odd, we have
$$
\text{\em rank}(\sigma_{d,n}) = \sum_{r=0}^{(d-1)/2}\binom{n}{r}.
$$
\end{thm}
\begin{proof}
We have a lower bound
$$\text{rank}(\sigma_{d,n}) \geq \dim_{\mathbb{C}}S'/(\sigma_{d-1,n-1})^{\bot}
= \sum_r \text{Hilb}(S'/(\sigma_{d-1,n-1})^{\bot},r).$$
By Corollary \ref{hilb}, we get
$$
\sum_r \text{Hilb}(S'/(\sigma_{d-1,n-1})^{\bot},r) \,\,=\,\, \sum_{r=0}^k \binom{n-1}{r} + \sum_{r=k+1}^{n-1} \binom{n-1}{d-r}
$$
$$
=\binom{n-1}{0} + \sum_{i=r}^k \left( \binom{n-1}{r} + \binom{n-1}{r-1} \right)
\,\,=\,\,1+\sum_{r=1}^k \binom{n}{r}
$$
and it is coincide to the upper bound given in Corollary \ref{bound1}.
\end{proof}
\begin{rmk}\em
One may ask for a power sum decomposition of a form over real field. The smallest number of required real linear forms is called the {\em real rank} of given form. Since all the coefficients in \eqref{psdec} are real, in fact rational, Theorem \ref{main} also holds for the real rank.
\end{rmk}

\section{Even degree}
For even $n=d=2k$, The equation \eqref{monomial} can be written in symmetric format as
\begin{align*}
2^{n-1}n! \cdot \sigma_{n,n} &= \sum_{I \subset [n], |I| < k} (-1)^{|I|}
(\delta(I,1) x_1 + \delta(I,2) x_2 + \dots +\delta(I,n) x_n)^n \\
&+ \sum_{I \subset [n], |I| = k}  \frac{(-1)^k}{2}
(\delta(I,1) x_1 + \delta(I,2) x_2 + \dots +\delta(I,n) x_n)^n.
\end{align*}
However it is hard to generalize this expression to elementary symmetric polynomials directly, while it is in fact possible (Remark \ref{direct}). Instead, we can easily get a power sum decomposition using Theorem \ref{decomp}.
\begin{thm}\label{decomp:odd}
Let $d=2k$ be even, and $n > d$. Then the elementary symmetric polynomial $\sigma_{d,n}$ admits the power sum decomposition
$$
2^d(n-d)d! \cdot \sigma_{d,n} =
$$
\begin{equation}\label{psdec:odd}
\sum_{I \subset [n], |I| \leq k} (-1)^{|I|} \binom{n-k-|I|-1}{k-|I|} (n-2|I|) (\delta(I,1) x_1 + \delta(I,2) x_2 + \dots +\delta(I,n) x_n)^d
\end{equation}
where $\delta(I,i) = -1 $ if $i \in I$, or $1$ otherwise.
\end{thm}
\begin{proof}
Note that
$$
\left(\frac{\partial}{\partial x_1} + \cdots + \frac{\partial}{\partial x_n}\right) \sigma_{d+1,n}
=(n-d) \sigma_{d,n}.
$$
Applying \eqref{psdec} to the left side, we get the required decomposition.
\end{proof}
\begin{rmk}\label{direct}\em
After knowing the decomposition \eqref{psdec:odd}, Theorem \ref{decomp:odd} can be proven inductively as the proof of Theorem \ref{decomp}.
\end{rmk}
\begin{rmk}\em
If we take $x_i=1$ for all $i$ in \eqref{psdec:odd} then we get \eqref{comb}.
\end{rmk}

Since the number of summands are not equal to the lower bound given by Corollary \ref{hilb} unless $n=d$, we do not know whether the expression \eqref{psdec:odd} is minimal.
\begin{cor}\label{bound2}
For $d$ even, the rank of $\sigma_{d,n}$ is bounded by
$$
\left(\sum_{r=0}^{d/2}\binom{n}{r}\right) - \binom{n-1}{d/2}\leq
\text{\em rank}(\sigma_{d,n}) \leq \sum_{r=0}^{d/2} \binom{n}{r}.$$
\end{cor}

We close this section by giving a difficulty of an improvement of Corollary \ref{bound2}. If $d=4$ and $n=5$, for instance, then the bound is $10 \leq \text{rank}(\sigma_{4,5}) \leq 16$. It is easy to see that $\text{rank}(\sigma_{4,5}) \geq 11$ by observing $\text{Hilb}(S/(\sigma_{4,5})^{\bot}, 2)=10$ and $x_1^2, \dots, x_5^2 \in (\sigma_{4,5})^{\bot}$. Both lower bounds given in \cite{LT} and \cite{RS} still give $11$. The following proposition suggests that it will be tough to find a decomposition with less than $15$ pure powers, even if there exists. It is checked by brute force using {\em Macaulay2} \cite{M2}.
\begin{prop}
Let $I$ be the defining ideal of $15$ points among $16$ points of the form $(\pm 1, \pm 1, \pm 1, \pm 1, \pm 1) \in \mathbb{P}^4$. Then $I$ is not contained in $(\sigma_{4,5})^{\bot}$.
\end{prop}
It means that $\sigma_{4,5}$ can not be written as a linear combination of $15$ (or less) polynomials of the form $(\pm x_1 \pm x_2 \pm x_3 \pm x_4 \pm x_5)^4$.

\medskip \bigskip \bigskip

\noindent
{\bf Acknowledgements.}\smallskip \\
The author thanks to Joseph Landsberg and Bernd Sturmfels for their encouragement on this paper.
He also greatful to Peter B{\"u}rgisser for his useful comment which makes \hbox{Lemma \ref{fullrank}} into a simple form.
This work was initiated by a mentoring program named ``NIMS school on applied algebraic geometry'', held in the spring of 2015 and supported by NIMS and SRC-GAIA, with author's mentor Youngho Woo.

\begin{small}

\end{small}

\bigskip

\noindent
\footnotesize {\bf Author's addresses:}

\noindent Hwangrae Lee, 
Pohang University of Science and Technology, Korea,
 {\tt meso@postech.ac.kr}
 
\end{document}